\documentclass[a4paper]{amsart}

\usepackage[latin1]{inputenc}
\usepackage{amssymb,amsfonts,amsmath,url}
\usepackage[all]{xy}

\usepackage{graphicx}
\usepackage{latexsym,longtable,amscd,booktabs,array}
\usepackage{lscape}
\usepackage{color}

\theoremstyle{plain}
\newtheorem{theorem}{Theorem}
\newtheorem{lemma}{Lemma}

\newtheorem{proposition}{Proposition}

\newtheorem{remark}{Remark}

\newcommand{\QQ}{{\mathbb Q}}
\newcommand{\Qbar}{\overline{\mathbb Q}}
\newcommand{\ZZ}{{\mathbb Z}}
\newcommand{\vs}{\vspace{.15cm}}
\newcommand{\VS}{\vspace{.3cm}}

\newcommand{\TE}{{E(\mathbb{Q})_{\rm tors}}}

\title{Galois Theory, discriminants and torsion subgroup of elliptic curves}

\author{Irene Garc\'{\i}a--Selfa}
\address{Departamento de Matem\'aticas, Universidad de Huelva. Facultad de Ciencias Experimentales. Campus de \lq\lq El Carmen", Avenida de las Fuerzas Armadas, s/n. 21071 Huelva (Spain).}
\email{irene.garcia@dmat.uhu.es} 

\author{Enrique Gonz\'alez--Jim\'enez}
\address{Departamento de Matem\'aticas, Universidad Aut\'onoma de Madrid. 28049 Madrid  (Spain).}
\email{enrique.gonzalez.jimenez@uam.es}

\author{Jos\'e M. Tornero}
\address{Departamento de \'Algebra, Universidad de Sevilla. P.O. 1160. 41080 Sevilla (Spain).}
\email{tornero@us.es}

\date{\today}

\subjclass[2000]{11G05 (primary); 11R16 (secondary).}
\keywords{Elliptic curves, Cubic equations, Torsion subgroup, Galois theory}

\begin{document}

\begin{abstract}
We find a tight relationship between the torsion subgroup and the image of the mod $2$  Galois representation associated to an elliptic curve defined over the rationals.  This is shown using some characterizations for the squareness of the discriminant of the elliptic curve.
\end{abstract}

\maketitle


\section{Introduction.}


In what follows we will denote by $C_n$ and $S_n$ the cyclic group of order $n$ and the symmetric group acting on $n$ elements, respectively. 

Let $E$ be an elliptic curve defined over $\QQ$. Let $p$ be a prime number and let $E[p]$ be the group of points of order $p$ on $E(\Qbar)$, where $\Qbar$ denotes an algebraic closure of $\QQ$. The action of the absolute Galois group ${\rm G}_\QQ={\rm Gal}(\Qbar/\QQ)$ on $E[p]$ defines a mod $p$ Galois representation
$$
\rho_{E,p}:{\rm G}_\QQ\rightarrow {\rm Aut}(E[p])\cong {\rm GL}_2(\mathbb{F}_p).
$$
Let $\QQ(E[p])$ be the number field generated by the coordinates of the points of $E[p]$. Therefore, the Galois extension $\QQ(E[p])/\QQ$ has Galois group
$$
{\rm Gal}(\QQ(E[p])/\QQ)\cong \rho_{E,p}({\rm G}_\QQ)
$$
For $p=2$ it is known that $\rho_{E,2}({\rm G}_\QQ)$ can be determined in terms of the discriminant $\Delta(E)$ and $E(\QQ)[2]$, the points of order $2$ defined over the rationals (cf. \cite{serre2,serre1,serre-tate}):
\begin{equation}\label{serre}
\rho_{E,2}({\rm G}_\QQ)\cong\left\{
\begin{array}{ccl}
S_3 & & \mbox{if $\sqrt{\Delta(E)}\not\in\QQ$ and $\#E(\QQ)[2]=1$},\\ 
C_3 & & \mbox{if $\sqrt{\Delta(E)}\in\QQ$ and $\#E(\QQ)[2]=1$},\\ 
C_2 & & \mbox{if $\sqrt{\Delta(E)}\not\in\QQ$ and $\#E(\QQ)[2]>1$},\\ 
\{{\rm id}\} & & \mbox{if $\sqrt{\Delta(E)}\in\QQ$ and $\#E(\QQ)[2]>1$}.\\ 
\end{array}
\right.
\end{equation}
Note that ${\rm GL}_2(\mathbb{F}_2)\cong S_3$, the non-split Cartan subgroup of ${\rm GL}_2(\mathbb{F}_2)$ is isomorphic to $C_3$ and the conjugated Borel subgroup of ${\rm GL}_2(\mathbb{F}_2)$ is isomorphic to $C_2$.

An elliptic curve $E$ defined over the rationals has always an integral short Weierstrass form. That is, $E$ has a model of the form 
$$
E: \ Y^2 = X^3+AX+B, \mbox{ with } A,B \in \ZZ.
$$ 
Then the discriminant of this model is $\Delta(E) = -2^4 (4A^3+27B^2)$. Any change of variables over the rationals preserving this short form is of the type $(x,y)=(u^2x',u^3y')$ with $u \in \QQ, \; u\ne 0$. Therefore, if $E'$ is the curve obtained after such a change, we have $u^{12} \Delta(E') = \Delta(E)$. Then the squareness of the discriminant of $E$ does not depend on the short model of $E$ but on $E$ itself.

Our aim is finding if there is a relationship between the torsion group of $E(\QQ)$ (noted $\TE$ in what follows), the discriminant of $E$ and $\rho_{E,2}({\rm G}_\QQ)$. 

Assume $E$ is an elliptic curve which has a non--trivial torsion subgroup. Taking into account Mazur's exhaustive classification \cite{MazurIHES, Mazur}, the possible structures of $\TE$ are $C_n$ for $n=2\dots10,12$ and $C_2\times C_{2n}$ for $n=1\dots 4$. The easiest cases are those in which the order of $\TE$ is even and they will be treated at section \ref{s2}. 

The four remaining cases, $\TE=C_n$, with $n=3,5,7,9$ will be treated separately at sections  \ref{s3},  \ref{s5},  \ref{s7},  \ref{s9}, respectively. In these cases, thanks to (\ref{serre}), the squareness of $\Delta(E)$ determines the image of $\rho_{E,2}$. We will prove that there are no elliptic curves over the rationals with square discriminant and points of order $5,7$ and $9$ respectively and we will give a parametrization of the elliptic curves with square discriminant and a point of order $3$.

Section \ref{odd} consists of the necessary background for elliptic curves with points of odd order. 

Before stating the main theorems at the last section, we will give a parametrization of all elliptic curves over the rationals having square discriminant at section \ref{s0}. Some remarks on the case of trivial torsion will be given there too. 

At section \ref{s_thm}, we will state the main theorems of this paper whose proofs will have been stablished by then. 

Finally, in an appendix, we give a complete parametrization of the integer solutions of the Diophantine equations $x^2+3y^2=4z^3$. These solutions will be needed at sections \ref{s3} and \ref{s0}.

\subsection*{Acknowledgements}
We would like to thank to F. Beukers, H. Cohen, H. Darmon, L. Dieulefait, J. Gonz\'alez and J. Fern\'andez for useful comments during the preparation of this paper. Special thanks are due to F. Beukers for suggesting the complete parametrization of the Diophantine equation $x^2+3y^2=4z^3$. 

First and third authors were supported partially by FQM--218 and P08--FQM--03894 (Junta de Andaluc\'{\i}a) and MTM 2007--66929 (Ministerio de Educaci\'on y Ciencia, Spain). Research of the second author was supported in part by grant MTM 2006--10548 (Ministerio de Educaci\'on y Ciencia, Spain) and CCG07--UAM/ESP--1814 (Universidad Aut\'onoma de Madrid -- Comunidad de Madrid, Spain).

\section{The even case.}\label{s2}

Let $E:Y^2=F(X)=X^3+AX+B$ be an elliptic curve over $\QQ$ such that $E(\QQ)[2]$ has an even positive number of points. Therefore, by (\ref{serre}),  the squareness of $\Delta(E)$ determines the image of the mod $2$ Galois representation attached to $E$.  By definition $\Delta(E)=2^4(\alpha_1-\alpha_2)^2(\alpha_1-\alpha_3)^2(\alpha_2-\alpha_3)^2$, where $\alpha_1,\alpha_2,\alpha_3$ are the roots of $F(X)$. Then if  $\TE$ is non--cyclic we have that $\alpha_1,\alpha_2,\alpha_3\in\QQ$ and $\Delta(E)$ is a square over $\QQ$. Meanwhile, if $\TE$ is cyclic then there is only a point of order $2$ on $E(\QQ)$ and therefore $F(X)=(X-a)(X^2+aX+b)$ where $a,b\in\QQ$ satisfy $A=b-a^2$, $B=-ab$ and $a^2-4b$ is a non-square over $\QQ$. Since $\Delta(E)=2^4\Delta(F)$ we have $\Delta(E)=2^4(a^2-4b)(2a^2+b)^2$ is not a square in $\QQ$. This proves the following:

$$
\rho_{E,2}({\rm G}_\QQ)\cong\left\{
\begin{array}{ccl}
C_2 & & \mbox{if $\#E(\QQ)[2]=2$},\\ 
\{{\rm id}\} & & \mbox{if $\#E(\QQ)[2]=4$}.\\ 
\end{array}
\right.
$$

\section{Families of elliptic curves with a torsion point of odd order.}\label{odd}

In this section we are going to introduce the necessary background related to elliptic curves defined over the rationals with a point of prescribed odd order. There are well-known rational parametrizations for the modular curve $X_1(N)$ with $N \in \{3, 5, 7, 9\}$ (see e.g. Kubert \cite{Kubert}). Therefore these parametrizations give us families of elliptic curves defined over $\QQ$ with a point of order $N \in \{3, 5, 7, 9\}$. 

An old characterization of elliptic curves containing a rational point of order 3 is given by the Hessian form. Nevertheless we are going to use a new one (cf. \cite{GT}) since this will fit better our purposes. 

Let us introduce the construction given in \cite{GT}; every elliptic curve with a rational point of order $3$ can be written in the following form:
$$
E_3(\alpha, \beta)\,:\, Y^2  = X^3+(27\alpha^4+6 \alpha \beta)X+ \beta^2-27 \alpha^6,\quad \alpha,\beta\in \ZZ.
$$

For the remaining cases that will be used below, an analogous expression can be achieved by means of the Tate normal form \cite{Kubert}:
$$
\mathcal{T}ate(b,c)\,:\,Y^2 + (1-c) XY  - b Y = X^3 - bX^2, \quad b,c\in \QQ^*.
$$

Denote by $E_n(\alpha)$ the one-parameter family of curves having a rational point of order $n$. Then
\begin{eqnarray*}
E_5(\alpha)&=&\mathcal{T}ate( \alpha, \alpha),\\
E_7(\alpha)&=&\mathcal{T}ate(\alpha^2(\alpha - 1),\alpha(\alpha - 1))),\\
E_9(\alpha)&=&\mathcal{T}ate(\alpha^2(\alpha-1)(\alpha(\alpha-1)+1),\alpha^2(\alpha-1)).
\end{eqnarray*}

Now we can take the above equations to a short Weierstrass form and find the parametric family containing all elliptic curves with points of order $5$, $7$ and $9$. The actual families are:
\begin{eqnarray*}
E_5 (\alpha): Y^2 &=&  \displaystyle X^3 -27\left(\alpha^4 - 12\alpha^3 + 14\alpha^2 + 12\alpha + 1\right) X \qquad\qquad\qquad\qquad\quad\\ 
&& \displaystyle + \,54\left(\alpha^2 + 1\right)\left(\alpha^4 - 18 \alpha^3 + 74 \alpha^2 + 18 \alpha + 1\right) 
\end{eqnarray*}
\begin{eqnarray*}
E_7 (\alpha): Y^2 &=&  \displaystyle X^3 -27\left(\alpha^8 - 12\alpha^7 + 42\alpha^6 - 56\alpha^5 + 35\alpha^4 - 14\alpha^2 + 4\alpha + 1\right) X \\
&& \displaystyle +\, 54\left(\alpha^{12} - 18\alpha^{11} + 117\alpha^{10} - 354\alpha^9 + 570\alpha^8 - 486\alpha^7 + \right.\\
&& \displaystyle \left. 273\alpha^6 - 222\alpha^5 + 174\alpha^4 - 46\alpha^3 - 15\alpha^2 + 6\alpha + 1\right)
\end{eqnarray*}
\begin{eqnarray*}
E_9 (\alpha): Y^2 &=& \displaystyle X^3 -27 \left(\alpha^3 - 3 \alpha^2 + 1\left)\right(\alpha^9 - 9 \alpha^8 + 27 \alpha^7 - 48 \alpha^6 + 54 \alpha^5 -\right.\\
&& \displaystyle \left. 45 \alpha^4 + 27 \alpha^3 - 9 \alpha^2 + 1\right)X\\
&& \displaystyle +\, 54\left(\alpha^{18} - 18\alpha^{17} + 135\alpha^{16} - 570\alpha^{15} + 1557\alpha^{14} - 2970\alpha^{13} +\right.\\
&& \displaystyle \left. 4128\alpha^{12} - 4230\alpha^{11} + 3240\alpha^{10} - 2032\alpha^9 + 1359\alpha^8 - 1080\alpha^7 +\right.\\
&& \displaystyle \left.  735\alpha^6 - 306\alpha^5 + 27\alpha^4 + 42\alpha^3 - 18\alpha^2 + 1\right)
\end{eqnarray*}

Therefore, an elliptic curve $E$ defined over $\QQ$ with a rational point of order $n=3$ (resp.  $n=5$, $7$ or $9$) is $\QQ$--isomorphic to  $E_3(\alpha,\beta)$ (resp. $E_n (\alpha)$ for $n=5,7$ or $9$) for some $\alpha,\beta\in\ZZ$ (resp. $\alpha\in\QQ$).

These kind of arguments have proved fruitful in the last years, as a number of results have appeared based on them \cite{BI,Ingram,GOT,GT2}. Now we can write the discriminant $\Delta_n$  for the above elliptic curves $E_n$, to obtain
\begin{eqnarray*}
\Delta_3 (\alpha,\beta)& = &  - 2^4 \cdot 3^3 \cdot (5\alpha^3 + \beta)(9\alpha^3+\beta)^3 \\
\Delta_5 (\alpha)& = & 2^{12} \cdot 3^{12} \cdot \alpha^5 (\alpha^2-11 \alpha -1)\\
\Delta_7 (\alpha)& = &2^{12} \cdot 3^{12} \cdot \alpha^7(\alpha-1)^7(\alpha^3-8\alpha^2+5\alpha+1)\\
\Delta_9 (\alpha) &=& 2^{12} \cdot 3^{12} \cdot \alpha^9 (\alpha -1)^9 (\alpha^2-\alpha+1)^3(\alpha^3-6\alpha^2+3\alpha+1) 
\end{eqnarray*}

In the following section we will study the rationality of the square root of the above discriminants to decide whether the corresponding Galois group is $C_3$ or $S_3$.

\section{The case $n=3$}\label{s3}

Elliptic curves with points of order three must yield a discriminant with the form $
\Delta_3(\alpha,\beta)$ for some $\alpha,\beta\in \ZZ$. So, in order to find an elliptic curve $E$ having square discriminant we are bound to find integral solutions to the equation
$$
\omega^2 = - 3 \cdot (5\alpha^3 + \beta)(9\alpha^3+\beta).
$$
Let us denote $g=\gcd(5\alpha^3 + \beta,9\alpha^3+\beta)$. This necessarily leads to
$$
5\alpha^3 + \beta=\pm gu^2,\pm 3gv^2 \mbox{ and } 9\alpha^3+\beta =  \mp 3 gv^2,\mp gu^2, \mbox{ respectively} 
$$
for some integers $u$ and $v$.  Solving the above Diophantine systems of equations is equivalent to finding the integer solutions to  
$$
x^2+3y^2=4 z^3,
$$
where $(x,y,z)=(ug^2,vg^2,\mp \alpha g)$. Thus, for the first two systems we obtain that the elliptic curves $E_3(\alpha,\beta)$ have square discriminant for:
\begin{equation}\label{fam1}
(\alpha,\beta)  =  \left(\mp \frac{z}{g} , \pm \frac{x^2+5z^3}{g^3}\right),
\end{equation}
and for the last two systems we obtain
\begin{equation}\label{fam2}
(\alpha,\beta)  =  \left(\mp \frac{z}{g} , \pm \frac{3y^2+5z^3}{g^3}\right).
\end{equation}

In all those cases $x,y,z\in \ZZ$ satisfies $x^2+3y^2=4z^3$.  At the appendix, we give parametrizations of all integer solutions of the Diophantine equation $x^2+3y^2=4 z^3$ at Lemma \ref{ecu4} in terms of parameters $(a,b,c,d)$. Then we can clear denominators and obtain $\QQ$-isomorphic elliptic curves  
$$
E^{(i)}(a,b,c,d):Y^2=P^{(i)}(a,b,c,d)(X)
$$
attached to the parametrization $(i)$, for $i=1,2$, where:
{\small
$$
\begin{array}{l}
P^{(1)}(a,b,c,d)(X)= X^3-9 (c^2 + c d + d^2)^3 (a^2 + a b + b^2)
(3 a^6 c^2 + 3 a^6 c d + a^6 d^2 + 9 a^5 b c^2  - 3 a^5 b c d \\
\quad - 3 a^5 b d^2 - 30 a^4 b^2 c d - 15 a^3 b^3 c^2 - 15 a^3 b^3 c d + 25 a^3 b^3 d^2 +30 a^2 b^4 c d  + 30 a^2 b^4 d^2 + 9 a b^5 c^2 \\ \quad+ 21 a b^5 c d+ 9 a b^5 d^2 + 3 b^6 c^2 + 3 b^6 c d  + b^6 d^2)X \\
\quad + 9 (c^2 + c d + d^2)^4(6 a^{12} c^4 + 12 a^{12} c^3 d + 12 a^{12} c^2 d^2 + 6 a^{12} c d^3 + a^{12} d^4 + 36 a^{11} b c^4 \\
\quad + 36 a^{11} b c^3 d + 18 a^{11} b c^2 d^2 - 6 a^{11} b c d^3 - 6 a^{11} b d^4 +  72 a^{10} b^2 c^4 - 54 a^{10} b^2 c^3 d -72 a^{10} b^2 c^2d^2 \\
\quad - 78 a^{10} b^2 c d^3 - 18 a^{10} b^2 d^4 + 30 a^9 b^3 c^4 -  318 a^9 b^3 c^3 d - 102 a^9 b^3 c^2 d^2 - 132 a^9 b^3 c d^3 - 4 a^9 b^3 d^4 \\
\quad - 81 a^8 b^4 c^4 - 378 a^8 b^4 c^3 d 
  + 162 a^8 b^4 c^2 d^2 - 252 a^8 b^4 c d^3 + 45 a^8 b^4 d^4 - 108 a^7 b^5 c^4 - 108 a^7 b^5 c^3 d\\
  \quad  +  108 a^7 b^5 c^2 d^2
  - 576 a^7 b^5 c d^3 + 216 a^7 b^5 d^4 - 72 a^6 b^6 c^4 - 144 a^6 b^6 c^3 d - 468 a^6 b^6 c^2 d^2 \\
  \quad - 396 a^6 b^6 c d^3 
 + 600 a^6 b^6 d^4 - 108 a^5 b^7 c^4 - 324 a^5 b^7 c^3 d - 216 a^5 b^7 c^2 d^2 + 684 a^5 b^7 c d^3 \\
 \quad + 900 a^5 b^7 d^4
 - 81 a^4 b^8 c^4 + 54 a^4 b^8 c^3 d + 810 a^4 b^8 c^2 d^2 +  1386 a^4 b^8 c d^3 + 756 a^4 b^8 d^4 + 30 a^3 b^9 c^4 \\
 \quad + 438 a^3 b^9 c^3 d + 1032 a^3 b^9 c^2 d^2 + 1002 a^3 b^9 c d^3 + 374 a^3 b^9 d^4 + 72 a^2 b^{10} c^4 +  
  342 a^2 b^{10} c^3 d \\
  \quad + 522 a^2 b^{10} c^2 d^2 + 384 a^2 b^{10} c d^3 + 114 a^2 b^{10} d^4 + 36 a b^{11} c^4 + 
 108 a b^{11} c^3 d + 126 a b^{11} c^2 d^2 \\
 \quad + 78 a b^{11} c d^3 +  18 a b^{11} d^4 + 6 b^{12} c^4 + 12 b^{12} c^3 d 
 + 12 b^{12} c^2 d^2 + 6 b^{12} c d^3 + b^{12} d^4)
 \\[.2cm]
P^{(2)}(a,b,c,d)(X)=X^3-3(c^2 + c d + d^2)^3(a^2 + a b + b^2)(a^6 c^2 + a^6 c d + 7 a^6 d^2 + 3 a^5 b c^2 + 39 a^5 b c d \\
 \quad+ 39 a^5 b d^2 + 60 a^4 b^2 c^2 + 150 a^4 b^2 c d + 60 a^4 b^2 d^2 + 115 a^3 b^3 c^2 + 115 a^3 b^3 c d - 5 a^3 b^3 d^2 + 60 a^2 b^4 c^2 \\
 \quad- 30 a^2 b^4 c d - 30 a^2 b^4 d^2 + 3 a b^5 c^2 - 33 a b^5 c d + 3 a b^5 d^2 + b^6 c^2 + b^6 c d + 7 b^6 d^2)X\\
\quad  -(c^2 + c d + d^2)^4
(2 a^{12} c^4 + 4 a^{12} c^3 d - 24 a^{12} c^2 d^2 - 26 a^{12} c d^3 - 37 a^{12} d^4 + 12 a^{11} b c^4 - 156 a^{11} b c^3 d \\
 \quad- 414 a^{11} b c^2 d^2 - 534 a^{11} b c d^3 - 366 a^{11} b d^4 - 
228 a^{10} b^2 c^4 - 1446 a^{10} b^2 c^3 d - 2880 a^{10} b^2 c^2 d^2 \\
 \quad- 3246 a^{10} b^2 c d^3 - 1434 a^{10} b^2 d^4 - 1250 a^9 b^3 c^4 - 5902 a^9 b^3 c^3 d - 10758 a^9 b^3 c^2 d^2 - 
9508 a^9 b^3 c d^3\\
 \quad - 2876 a^9 b^3 d^4 - 4059 a^8 b^4 c^4 - 16002 a^8 b^4 c^3 d - 23382 a^8 b^4 c^2 d^2 - 14868 a^8 b^4 c d^3 - 2925 a^8 b^4 d^4 \\
 \quad- 8604 a^7 b^5 c^4 - 
25740 a^7 b^5 c^3 d - 26676 a^7 b^5 c^2 d^2 - 10944 a^7 b^5 c d^3 - 936 a^7 b^5 d^4 - 11112 a^6 b^6 c^4\\
 \quad - 22224 a^6 b^6 c^3 d - 13428 a^6 b^6 c^2 d^2 - 2316 a^6 b^6 c d^3 + 
480 a^6 b^6 d^4 - 8604 a^5 b^7 c^4 - 8676 a^5 b^7 c^3 d \\
 \quad- 1080 a^5 b^7 c^2 d^2 + 396 a^5 b^7 c d^3 + 468 a^5 b^7 d^4 - 4059 a^4 b^8 c^4 - 234 a^4 b^8 c^3 d + 
270 a^4 b^8 c^2 d^2 \\
 \quad- 126 a^4 b^8 c d^3 + 504 a^4 b^8 d^4 - 1250 a^3 b^9 c^4 + 902 a^3 b^9 c^3 d - 552 a^3 b^9 c^2 d^2 + 698 a^3 b^9 c d^3 + 526 a^3 b^9 d^4 \\
 \quad- 228 a^2 b^{10} c^4 + 534 a^2 b^{10} c^3 d + 90 a^2 b^{10} c^2 d^2 + 912 a^2 b^{10} c d^3 + 150 a^2 b^{10} d^4 + 12 a b^{11} c^4 + 204 a b^{11} c^3 d \\
 \quad+ 126 a b^{11} c^2 d^2 + 
222 a b^{11} c d^3 - 78 a b^{11} d^4 + 2 b^{12} c^4 + 4 b^{12} c^3 d - 24 b^{12} c^2 d^2 - 26 b^{12} c d^3 - 37 b^{12} d^4)
\end{array}
$$
}
with the following discriminants
\begin{eqnarray*}
\Delta \left( E^{(1)} \right) & = & 2^4 3^6 (c^2 + c d + d^2)^8 (a^3 d + 3 a^2 b c + 3 a^2 b d + 3 a b^2 c - b^3 d)^6\\
& & (2 a^3 c + a^3 d + 3 a^2 b c - 3 a^2 b d - 3 a b^2 c - 6 a b^2 d - 2 b^3 c - b^3 d)^2\\[.2cm]
\Delta \left( E^{(2)} \right) & = & 2^4 3^4 (c^2 + c d + d^2)^8 (a^3 d + 3 a^2 b c + 3 a^2 b d + 3 a b^2 c - b^3 d)^2\\
& & (2 a^3 c + a^3 d + 3 a^2 b c - 3 a^2 b d - 3 a b^2 c - 6 a b^2 d - 2 b^3 c - b^3 d)^6
\end{eqnarray*}

Therefore we have proved the following result:

\begin{proposition}\label{3_GT}
Let $E$ be an elliptic curve defined over $\QQ$ with a rational point of order $3$ such that $\sqrt{\Delta(E)}\in\QQ$ . Then there exist $a,b,c,d\in\ZZ$ such that $E$ is $\QQ$-isomorphic to either $E^{(1)}(a,b,c,d)$ or $E^{(2)}(a,b,c,d)$. 
\end{proposition}

\section{The case $n=5$}\label{s5}

Let us have a look at the case $n=5$. If we throw away quadratic factors in $\Delta_5 (\alpha)$ we will find out that curves $E$ with points of order $5$, for which $\sqrt{\Delta(E)} \in \QQ$ are parametrized by the affine rational points of the elliptic curve
$$
\mathcal{D}_5\,:\, z^2 = \alpha (\alpha^2-11 \alpha -1),
$$
where the discriminant of the right--hand side polynomial is, remarkably, $5^3$.

This is a well--known elliptic curve, in fact is $\QQ$-isogenous to the modular curve $X_0(20)$. The elliptic curve $\mathcal{D}_5 $ is denoted by \verb+20A4+ in Cremona's tables \cite{cremona} or \verb+20C+ in Antwerp tables \cite{antwerp}. Looking on that tables, or using a computer algebra package like \verb+SAGE+ or \verb+MAGMA+ (\cite{sage}, \cite{magma} resp.), we check that  $\mathcal{D}_5 (\QQ)=\{(0,0)\}\cup\{[0:1:0]\}\,$. Therefore the only affine rational point is $(0,0)$, which implies $\alpha = 0$. This precise value does not yield an elliptic curve, but a singular cubic on the family $E_5(\alpha)$. This proves the following result:

\vs

\begin{proposition}\label{5}
Let $E$ be an elliptic curve with $C_5\subset\TE$. Then $\sqrt{\Delta(E)} \notin \QQ$.
\end{proposition}

\section{The case $n=7$}\label{s7}

Move now to $n=7$, where the analogous argument to the case $n=5$ shows that curves with points of order seven for which $\sqrt{\Delta(E)} \in \QQ$ are parametrized by the affine rational points of the hyperelliptic curve
$$
\mathcal{D}_7\,:\, z^2 = \alpha (\alpha-1) (\alpha^3 - 8\alpha^2 + 5\alpha +1),
$$
where, by the way, we have that the discriminant  for the right--hand side polynomial is $7^4$. We have now a hyperelliptic curve of genus $2$, a much harder nut to crack; but we are lucky. Using \verb+MAGMA+ we obtain that the rank of the Jacobian of this genus $2$ curve is $0$, which makes it perfect for Chabauty's algorithm \cite{Chabauty}. This method computes the full list of points in the jacobian, then all rational points in the curve, which turn out to be $\mathcal{D}_7(\QQ)= \{(0,0),(1,0)\}\cup\{[0:1:0]\}$ . Again the affine rational points annhilate the discriminant of $E_7(\alpha)$ and hence we have proven the following result.

\vs

\begin{proposition}\label{7}
Let $E$ be an elliptic curve with $\TE \cong C_7$. Then $\sqrt{\Delta(E)} \notin \QQ$.
\end{proposition}

\section{The case $n=9$}\label{s9}

Finally,  the case $n=9$.  This  can also be dealt with in a similar way, but a little extra work is needed. Following the steps as the above sections, the hyperelliptic curve parametrizing curves with $\TE \cong  C_9$ and square discriminant is
$$
\mathcal{D}_9\,:\, z^2 = \alpha(\alpha-1)(\alpha^2-\alpha+1)(\alpha^3-6\alpha^2+3\alpha+1)\,.
$$

\noindent {\bf Lemma.} $\mathcal{D}_9(\QQ)=\{(0,0),(1,0)\}\cup\{[0:1:0]\}\,$. 

\VS

\begin{proof} Let $u\in \mbox{Aut}_{\QQ}(\mathcal{D}_9)$ defined by 
$$
u(X,Y)= \left( \frac{1}{1-X},\frac{Y}{(1-X)^{4}} \right). 
$$
We have that $u$ has order $3$ and Riemann-Hurwitz formulae tell us that the quotient curve $C/\langle u\rangle$ has genus $1$. In fact this curve is an elliptic curve defined over $\QQ$, since $(0,0)\in C(\QQ)$ and $u$ is defined over $\QQ$. We will denote this elliptic curve by $\mathcal{E}_9$. A Weierstrass equation for $\mathcal{E}_9$ is given by $v^2=u^3-27$ and the quotient morphism is given by:
\begin{eqnarray*}
\pi : \mathcal{D}_9 &\longrightarrow & \mathcal{E}_9 \\
       (\alpha,z) & \mapsto & \displaystyle (u,v)=\left(\frac{\alpha^3-3\alpha^2+1}{\alpha(\alpha-1)},\frac{z(\alpha^2-\alpha+1)}{\alpha^2(\alpha-1)^2}\right)
\end{eqnarray*}

Using \verb+SAGE+ or \verb+MAGMA+ we compute that $\mathcal{E}_9$ is $\QQ$-isogenous to the modular curve $X_0(36)$, and it is the elliptic curve denotes by \verb+36A3+ in Cremona's tables or \verb+36C+ in Antwerp tables. The Mordell-Weil group of this elliptic curve is:
$$
\mathcal{E}_9(\QQ)=\{(3,0), [0:1:0] \}\cong \ZZ/2\ZZ.
$$
Now, to compute the set $\mathcal{D}_9(\QQ)$ we just need to compute the preimages of the points of $\mathcal{E}_9(\QQ)$ by the quotien morphism that are defined over $\QQ$ and then we obtain the desired result.
\end{proof}

\vs
Then we have proved the following result:
\vs

\begin{proposition}\label{9}

Let $E$ be an elliptic curve with $\TE \cong  C_9$. Then $\sqrt{\Delta(E)} \notin \QQ$.
\end{proposition}

\section{The generic elliptic curve with square discriminant.}\label{s0}

Let $E\,:\, Y^2=X^3+A X+B$ be an elliptic curve with $A,B\in \ZZ$. Let us study when the discriminant $\Delta(E)$ is a square. This is equivalent to looking for integer solutions to the Diophantine equation 
$$
4 A^3+27 B^2=-C^2. 
$$

Making the change of variables $(x,y,z)=(C,3B,-A)$ we obtain that the integer solutions of the generalized Fermat equation  $x^2+3 y^2=4 z^3$ give us all the elliptic curves defined over the rationals with square discriminant. Lemma \ref{ecu4} from the appendix gives us a complete parametrization of the above Diophantine equation, which yields to the following elliptic curve:
\begin{eqnarray*}
E_{alt}(a,b,c,d)\,:\,Y^2 & =& X^3-3^4(c^2 + c d + d^2)(a^2 + a b + b^2)X\\
& & \qquad+ 3^{5}(c^2 + c d + d^2)(a^3 d + 3 a^2 b c + 3 a^2 b d + 3 a b^2 c - b^3 d)
\end{eqnarray*}
with discriminant
$$
\Delta(E_{alt}) = 2^4 3^{12} (c^2 + c d + d^2)^2(2 a^3 c + a^3 d + 3 a^2 b c - 3 a^2 b d - 3 a b^2 c - 6 a b^2 d - 2 b^3 c - b^3 d)^2
$$
Propositions \ref{5}, \ref{7} and \ref{9}, together with Section 2, tell us that the above elliptic curve has torsion subgroup either trivial, $C_3$ or non-cyclic. Then we have proved the following result:

\begin{proposition}
Let $E$ be an elliptic curve defined over $\QQ$ such that $\sqrt{\Delta(E)}\in\QQ$. Then there exist $a,b,c,d\in\ZZ$ such that $E$ is $\QQ$-isomorphic to $E_{alt}(a,b,c,d)$. Moreover, $\TE$ is either trivial, non-cyclic or $C_3$.
\end{proposition}

\begin{remark}{\rm
Let $P(X)$ be an irreducible polynomial with integer coefficients and degree $3$ such that the cubic number field attached to $P(X)$ is cyclic. Then the elliptic curve $E:Y^2=P(X)$ satisfies $\TE$ is either trivial or $C_3$. For example, let be  $P_m(X)=X^3+mX^2-(m+3)X+1\in \ZZ[X]$, $m\in\ZZ$ and $E_m:Y^2=P_m(X)$. The irreducible polynomial $P_m(X)$ defines a cubic fields $K_m$  that has been studied by several authors. This family has been called the {\it simplest cubic field}  (\cite{Shanks}). Its discriminant satisfies $\Delta(P_m)=(m^2+3 m +9)^2$, hence $K_m=\QQ(E[2])$ is cyclic and therefore $ \rho_{E_m,2}({\rm G}_\QQ)\cong {\rm Gal}(K_m/\QQ)= C_3$. Therefore $E_m(\QQ)_{\rm tors}$ is trivial or $C_3$. Moreover, it has been proved \cite{duquesne} that if $m^2+3 m +9$ is square-free, then $E_m(\QQ)_{\rm tors}$ is trivial.
}\end{remark}

\begin{remark}{\rm
We have checked on the extended Cremona's tables \cite{cremonaweb} of elliptic curves with conductor less than $130.000$. Among them, $452.724$ curves have torsion subgroup either trivial or $C_3$. At the table below appears the specific proportions of curves according to their torsion group and the squareness of their discriminant:\\
\begin{center}
\begin{tabular}{|c|c|c|}
\hline
$\TE$ & $\sqrt{\Delta(E)}\in\QQ$ & $\sqrt{\Delta(E)}\not\in\QQ$\\
\hline
$\{ \mathcal O \}$ & $0.00383$ & $0.9553$ \\
\hline
 $C_3$ & $0.00008$ & $0.0408$ \\
\hline
\end{tabular}
\end{center}
}\end{remark}

\section{Main theorems.}\label{s_thm}

To end up we will summarize our results in the following theorems.

\begin{theorem}\label{square}
Let $E$ be an elliptic curve defined over $\QQ$. Then 
\begin{enumerate}
\item  If $\TE$ is non--cyclic then $\sqrt{\Delta(E)}\in\QQ$.
\item  If $\TE\cong C_n$ for $n=2,4,5\dots 10,12$ then $\sqrt{\Delta(E)}\not\in\QQ$.
\item  $\TE\cong C_3$ and $\sqrt{\Delta(E)}\in\QQ$ if and only if there exist $a,b,c,d\in\ZZ$ such that $E$ is $\QQ$-isomorphic to either $E^{(1)}(a,b,c,d)$ or $E^{(2)}(a,b,c,d)$ and the corresponding polynomial $P^{(i)}(a,b,c,d)$) is irreducible. 
\item $\TE$ is trivial and $\sqrt{\Delta(E)}\in\QQ$ if and only if there exist $a,b,c,d\in\ZZ$ such that $E$ is $\QQ$-isomorphic to $E_{alt}(a,b,c,d)$ and $a,b,c,d\notin \mathcal{S}_2,\mathcal{S}_3$ where
\begin{eqnarray*}
\mathcal{S}_2 & = & \left\{  (a,b,c,d)\in\ZZ^4\,\Big | \,\Psi_2(a,b,c,d)(X)=\prod_{i=1}^3{(X-\alpha_i)}\,\,\mbox{such that $\alpha_1,\alpha_2,\alpha_3\in\ZZ$}\right\}\\
\mathcal{S}_3 & = & \left\{  (a,b,c,d)\in \ZZ^4\,\Big | \,\exists (\alpha,\beta)\in\QQ^2,\mbox{such that  
$
\left\{
\begin{array}{l}
\Psi_3(a,b,c,d)(\alpha)=0\\
\Psi_2(a,b,c,d)(\alpha)=\beta^2
\end{array}
\right.
$}\right\}
\end{eqnarray*}
and $\Psi_n(a,b,c,d)(X)$ denotes the $n$-division polynomial attached to $E_{alt}(a,b,c,d)$.
\end{enumerate}
\end{theorem}

\begin{theorem}
Let $E$ be an elliptic curve defined over $\QQ$. Then
\begin{enumerate}
\item $\TE$ is non--cyclic if and only if $\rho_{E,2}({\rm G}_\QQ)=\{{\rm id}\}$.
\item $\TE \cong  C_{2n}$ if and only if $\rho_{E,2}({\rm G}_\QQ)\cong C_2$.
\item If $\TE \cong  C_n$ for $n=5,7,9$, then  $\rho_{E,2}({\rm G}_\QQ)={\rm GL_2(\mathbb{F}_2)}$. 
\item If $\TE \cong  C_3$ then $\rho_{E,2}({\rm G}_\QQ)\cong C_3$ if and only if there exist $a,b,c,d\in\ZZ$ such that $E$ is $\QQ$-isomorphic to either $E^{(1)}(a,b,c,d)$ or $E^{(2)}(a,b,c,d)$. Otherwise, $\rho_{E,2}({\rm G}_\QQ)={\rm GL_2(\mathbb{F}_2)}$.
\item If $\TE$ is trivial then $\rho_{E,2}({\rm G}_\QQ)\cong C_3$ if and only if there exist $a,b,c,d\in\ZZ$ such that $E$ is $\QQ$-isomorphic to  $E_{alt}(a,b,c,d)$. Otherwise, $\rho_{E,2}({\rm G}_\QQ)={\rm GL_2(\mathbb{F}_2)}$.
\end{enumerate}
\end{theorem}

The next table summarizes part of the main results of this paper:
\VS
\begin{center}
\begin{tabular}{|ccc|ccc|ccc|}
\hline
& $\TE$ & &   & $\sqrt{\Delta(E)}\in \QQ$? &  & & $\rho_{E,2}({\rm G}_\QQ)$ &\\
\hline
&  $\{ \mathcal O \}$ & & & Yes / No & & & $C_3$ / $S_3$ &\\
\hline
&  $C_2$ & & &  No & & & $C_2$ &\\
\hline
&  $C_3$ & & & Yes / No & & & $C_3$ / $S_3$ &\\
\hline
&  $C_4$ & & &  No & & & $C_2$ &\\
\hline
&  $C_5$ & & &  No & & & $S_3$ &\\
\hline
&  $C_6$ & & &  No & & & $C_2$ &\\
\hline
&  $C_7$ & & &  No & & & $S_3$ &\\
\hline
&  $C_8$ & & &  No & & & $C_2$ &\\
\hline
&  $C_9$ & & &  No & & & $S_3$ &\\
\hline
&  $C_{10}$ & & &  No & & & $C_2$ &\\
\hline
&  $C_{12}$ & & &  No & & & $C_2$ &\\
\hline
&  $C_2\times C_2$ & & &  Yes & & & $\{{\rm id}\}$ &\\
\hline
&  $C_2\times C_4$ & & &  Yes & & & $\{{\rm id}\}$ &\\
\hline
&  $C_2\times C_6$ & & &  Yes & & & $\{{\rm id}\}$ &\\
\hline
&  $C_2\times C_8$ & & &  Yes & & & $\{{\rm id}\}$ &\\
\hline
\end{tabular}
\end{center}

\section{Appendix: The generalized Fermat equation $x^2+3y^2=4z^3$}\label{parametrizacion}

The generalized Fermat equation $A x^p+B y^q = C z^r$, where $A,B,C \in \ZZ^*$ and $p,q,r\in \ZZ_{>0}$, have been studied by several people for the last decades. Starting with the huge work on Diophantine equations due to L.J. Mordell  \cite{Mordell}. After that the main results are due to H. Darmon and A. Granville \cite{DG}, who proved that the generalized Fermat equation has infinite primitive solutions (i.e. $gcd(x,y,z)=1$) in the case $1/p+1/q+1/r>1$. Then F. Beukers \cite{Beukers} gave parametrizations for the solutions of this equation in the above case. H. Cohen in his huge new books about number theory calls this equation Super-Fermat equation and he also provides solutions for several cases.  In particular, H. Cohen \cite[Proposition 14.2.1(ii)]{Cohen} displays primitive solutions to the equation we are interested in. However, we would like to have all integer solutions, not just the primitive ones. The following parametrization was pointed out to us by F. Beukers:

\begin{lemma}\label{ecu4}
The  integer solutions of the equation $x^2+3y^2=4z^3$ are parametrized by the following family of four variables:
$$
\mathcal{F}\,:\,
\left\{
\begin{array}{rcl}
x & = & (c^2 + c d + d^2) (3 a^2 b (c - d) + a^3 (2 c + d) - b^3 (2 c + d) -  3 a b^2 (c + 2 d))\\
y & = & (c^2 + c d + d^2) (3 a b^2 c + a^3 d - b^3 d + 3 a^2 b (c + d))\\
z & = & (c^2 + c d + d^2) (a^2 + a b + b^2) 
\end{array}
\right.
$$
\end{lemma}

\begin{proof}
Let $(x,y,z)$ be an integer solution of the equation $x^2+3y^2=4z^3$. Over $\QQ(\sqrt{-3})$we have $(x+\sqrt{-3}y) (x-\sqrt{-3} y)=4z^3$. We are going to work over the ring of algebraic integers of $\QQ(\sqrt{-3})$, noted $\mathcal{O}=\ZZ[\rho]$ where $\rho=(1+\sqrt{-3})/2$. This ring is a P.I.D., where it is easy to check that $\gcd(x+\sqrt{-3}y,x-\sqrt{-3} y)=r\in\ZZ$, for all $x,y \in \ZZ$. Therefore $x+\sqrt{-3}y=r\cdot \mu$,  $x-\sqrt{-3}y=r\cdot \bar\mu$ with $\mu\in \mathcal{O}$ and $\gcd(\mu,\bar\mu)=1$. 

Let $p\neq 3$ be a non inert prime dividing $r$. Thus $p=\alpha\bar\alpha$ for some $\alpha\in\mathcal{O}$. Suppose that $\alpha$ divides $\mu$, then we have $r=r'p$ and $\mu=\alpha\mu'$ for some $r'\in \ZZ$ and $\mu'\in\mathcal{O}$. In other words, $x+\sqrt{-3}y=(r'\alpha\bar\alpha)(\alpha\mu')$. Therefore
$$
4 z^3=r^2\mu\bar\mu=(r')^2 p^3 \mu' \bar\mu',
$$
that is, $p$ divides $z$. Then we can remove all the primes as above obtaining $4w^3=s^2\gamma\bar \gamma$ such that $s,\gamma,\bar\gamma$ are pairwise coprimes. This yields $s=2t^3$ and $\gamma=\beta^3$ for some $t\in \ZZ$ and $\beta\in\mathcal{O}$. Now collecting all the factors back we obtain 
$$
x+\sqrt{-3}y = 2 (a+b\rho)^3 (c+d \rho)(c^2 + c d + d^2)
$$
for some $a,b,c,d\in \ZZ$. Its easy to check that $z= \mathcal{N}_{\mathcal{O}}((a+b \rho)(c+d \rho))$, where $\mathcal{N}_{\mathcal{O}}$ denotes the norm on ${\mathcal{O}}$. In order to attach the parametrization, we only have to expand the above expression and compute the coefficient of $\sqrt{-3}$ (corresponding to $y$) and the part in $\ZZ$ (corresponding to $x$).
\end{proof}

\end{document}